\documentclass[reqno]{amsart}
\usepackage{kotex}
\usepackage{amsmath}
\usepackage{amssymb}
\usepackage{amscd}
\usepackage[all]{xy}
\usepackage{color}
\usepackage{mathrsfs}
\usepackage{tikz-cd}

\newtheorem{theorem}{Theorem}[section]
\newtheorem{lemma}[theorem]{Lemma}
\newtheorem{corollary}[theorem]{Corollary}
\newtheorem{proposition}[theorem]{Proposition}
\newcommand{\acknowledgments}{\section*{Acknowledgments}}

\theoremstyle{definition}
\newtheorem{definition}[theorem]{Definition}
\newtheorem{example}[theorem]{Example}
\newtheorem{remark}[theorem]{Remark}

\begin{document}
\title[Multiplicity-free modules for a nil-DAHA]
{On finite-dimensional multiplicity-free irreducible modules for a nil-DAHA of type $(C_1^\vee,C_1)$}

\author[J. Park]{Jongyook Park}
\address{(Park) Department of Mathematics,
College of Natural Sciences,
Kyungpook National University, Daegu 41566,
Republic of Korea}
\email{jongyook@knu.ac.kr}

\author[J.-H. Lee]{Jae-Ho Lee}
\address{(Lee) Department of Mathematics and Statistics,
University of North Florida,
Jacksonville, FL 32224, USA}
\email{jaeho.lee@unf.edu}

\author[H. Baek]{Hyungtae Baek$^{\dag}$}
\address{(Baek) School of Mathematics,
Kyungpook National University, Daegu 41566,
Republic of Korea}
\email{htbaek5@gmail.com}

\thanks{$\dag$: Corresponding author}
\thanks{All authors contributed equally to this paper}

\begin{abstract}
Fix nonzero $r_0,r_1\in\mathbb{C}$.
Let $\widetilde{\mathcal H}$ denote a nil-DAHA of type
$(C_1^\vee,C_1)$ defined by generators $t_0,u_0,t_1,u_1$ and relations
$(t_i-r_i)(t_i-r_i^{-1})=0$ for $i\in\{0,1\}$,
$u_0^2=u_0$, $u_1^2=0$, and
$u_0t_0t_1u_1=0=t_1u_1u_0t_0$.
Set $A=u_0t_0$ and $B=t_1u_1$.
A finite-dimensional $\widetilde{\mathcal H}$-module is called
$(A,B)$-multiplicity-free, or simply multiplicity-free, if $A$ and $B$
are simultaneously diagonalizable and every nonzero common eigenspace
is one-dimensional.
We consider finite-dimensional irreducible multiplicity-free modules that
have a certain ordered basis, which we call an adapted block basis.
For $D\geq 1$, we construct a family of $2D$-dimensional
$\widetilde{\mathcal H}$-modules $E_D$, and for $D\geq 0$, we construct
a family of $(2D+1)$-dimensional $\widetilde{\mathcal H}$-modules $O_D$.
We determine which members of these families are multiplicity-free and
irreducible.
We prove that every finite-dimensional irreducible
$\widetilde{\mathcal H}$-module that is multiplicity-free and has an
adapted block basis is isomorphic to a module of the form $E_D$ or $O_D$.
We also determine when two members of the same family are
isomorphic.
\end{abstract}

\maketitle

\section{Introduction}

Double affine Hecke algebras (DAHAs) were introduced by Cherednik
\cite{Cherednik1992} and subsequently used in his proofs of Macdonald's
conjectures \cite{Cherednik1995AnnM,Cherednik1995InvM}.
Sahi extended the theory to the nonreduced affine root system of type
$(C_n^\vee,C_n)$ and related the corresponding DAHA to nonsymmetric
Koornwinder polynomials \cite{Sahi1999}.
In the rank-one case $n=1$, the symmetric and
nonsymmetric Koornwinder polynomials reduce to the Askey-Wilson and nonsymmetric Askey-Wilson polynomials, respectively
\cite[Sections~6.5,~6.6]{Macdonald2003}.
Thus the DAHA of type $(C_1^\vee,C_1)$ provides an algebraic framework
for studying the Askey-Wilson polynomials and their nonsymmetric
counterparts \cite{NoumiStokman2004}.
The finite-dimensional irreducible modules for the universal DAHA of type
$(C_1^\vee,C_1)$ were classified by Huang \cite{Huang}.
Nil-DAHAs arise from certain degenerations of DAHAs.
Cherednik and Orr studied the nil-DAHA for the root system $A_1$ and its
connections with symmetric and nonsymmetric $q$-Whittaker functions and
the $q$-Toda eigenvalue problem
\cite{CherednikOrr2012,CherednikOrr2013}.
They extended this work to reduced root systems
\cite{CherednikOrr2015}.
The present paper concerns finite-dimensional irreducible modules for a
rank-one nil-DAHA.

One motivation for this paper comes from the theory of $Q$-polynomial
distance-regular graphs.
In \cite{Lee2013}, Lee associated certain $Q$-polynomial
distance-regular graphs with finite-dimensional irreducible modules for
the universal DAHA of type $(C_1^\vee,C_1)$.
In \cite{Lee2017}, these modules were used to obtain nonsymmetric
$q$-Racah polynomials from such graphs.
These polynomials are finite counterparts of the nonsymmetric
Askey-Wilson polynomials.
Subsequently, in \cite{LeeTanaka}, Lee and Tanaka pursued
this approach in the setting of dual polar graphs and considered a
nil-DAHA of type $(C_1^\vee,C_1)$.
More specifically, let $\Gamma$ be a dual polar graph with diameter
$D\geq 3$, let $x$ be a vertex of $\Gamma$, and let $C$ be a maximal
clique containing $x$.
Using the pair $(x,C)$, they constructed a $2D$-dimensional irreducible
module for this nil-DAHA.
This module was used to define nonsymmetric dual $q$-Krawtchouk
polynomials and to derive their recurrence and orthogonality relations.
Moreover, with respect to a suitable basis, the generators are represented
by block diagonal matrices with $1\times 1$ and $2\times 2$ blocks.
The pattern of these blocks motivates the \emph{adapted block bases} used in
the present paper; see Definition~\ref{adapted basis}.

We now state a goal of this paper.
For this purpose, we first recall the definition of the nil-DAHA of type
$(C_1^\vee,C_1)$ considered here.
\begin{definition}[{\cite[Definition~5.1]{LeeTanakaSLC}}]\label{nil.H}
Fix nonzero $r_0,r_1\in\mathbb{C}$.
Let $\widetilde{\mathcal{H}}=\widetilde{\mathcal{H}}(r_0,r_1)$ denote
the unital associative $\mathbb{C}$-algebra generated by
$t_0,u_0,t_1,u_1$ subject to the relations
\begin{itemize}
\item[(i)] $(t_i-r_i)(t_i-r_i^{-1})=0$ for $i\in\{0,1\}$;
\item[(ii)] $u_0^2=u_0$;
\item[(iii)] $u_1^2=0$;
\item[(iv)] $u_0t_0t_1u_1=0=t_1u_1u_0t_0$.
\end{itemize}
We call $\widetilde{\mathcal{H}}$ a \emph{nil-DAHA of type
$(C_1^\vee,C_1)$}.\footnote{The definition of a nil-DAHA of type
$(C_1^\vee,C_1)$ given here differs from the one in
\cite[Definition~8.1]{LeeTanaka}. In fact, the algebra defined there is a
homomorphic image of $\widetilde{\mathcal H}$.}
\end{definition}
\noindent
Define $A,B\in\widetilde{\mathcal H}$ by $A=u_0t_0$ and $B=t_1u_1$.
We say that a finite-dimensional $\widetilde{\mathcal H}$-module $V$ is
\emph{$(A,B)$-multiplicity-free} if $A$ and $B$ are simultaneously
diagonalizable on $V$ and every nonzero common eigenspace of $A$ and $B$
is one-dimensional.
In the present paper, we classify the finite-dimensional irreducible
$\widetilde{\mathcal H}$-modules that are $(A,B)$-multiplicity-free and
have an adapted block basis.

We now summarize the results of this paper.
For $D\geq 1$, we construct a family of $2D$-dimensional
$\widetilde{\mathcal H}$-modules $E_D$, and for $D\geq 0$, we construct
a family of $(2D+1)$-dimensional $\widetilde{\mathcal H}$-modules
$O_D$ (Propositions~\ref{prop:E_D} and~\ref{prop:O_D}).
The actions of $t_0,u_0,t_1,u_1$ on these modules are given by block
matrices.
The bases used in these constructions are adapted block bases
(Propositions~\ref{prop:E_D,abb} and~\ref{prop:O_D,abb}).
We give conditions on the parameters under which these modules are
$(A,B)$-multiplicity-free.
Under these conditions, we prove that the modules are irreducible
(Theorems~\ref{even irreducible} and~\ref{odd irreducible}).
Finally, we show that every finite-dimensional irreducible
$\widetilde{\mathcal H}$-module that is $(A,B)$-multiplicity-free and
has an adapted block basis is isomorphic to a module of the form $E_D$
or $O_D$, and we determine the isomorphism classes within each family
(Theorem~\ref{thm:main}).
This classification is the main result of the paper.

This paper is organized as follows.
In Section~\ref{Sec:MFABB}, we introduce the notions
of $(A,B)$-multiplicity-free $\widetilde{\mathcal{H}}$-modules
and adapted block bases, and give a necessary condition for irreducibility.
In Section~\ref{Sec:ED,OD}, we construct the families $E_D$ and $O_D$ and discuss their multiplicity-freeness and irreducibility.
In Section~\ref{Sec:MR}, we prove the classification theorem and determine the isomorphism classes within the families $E_D$ and $O_D$.
In Section~\ref{Sec:ex}, we give three examples of 
finite-dimensional irreducible $\widetilde{\mathcal{H}}$-modules, including the module arising from a dual
polar graph in \cite{LeeTanaka}.

\section{Multiplicity-free modules and adapted block bases}\label{Sec:MFABB}

In this section, we introduce the notions of multiplicity-free modules and adapted block bases for finite-dimensional $\widetilde{\mathcal{H}}$-modules. 
We also give a necessary condition for an $\widetilde{\mathcal{H}}$-module with an adapted block basis to be irreducible.
Throughout this paper, we set
\begin{equation}\label{eq:A,B}
A=u_0t_0,\qquad B=t_1u_1.
\end{equation}
Observe that $AB=BA=0$ by Definition~\ref{nil.H}(iv).
Let $V$ be a finite-dimensional $\widetilde{\mathcal{H}}$-module.
For $\alpha, \beta \in \mathbb{C}$, define
$$
V_{\alpha, \beta} = \{v \in V \mid Av = \alpha v, Bv = \beta v\}.
$$
We call $V_{\alpha,\beta}$ the \emph{$(\alpha,\beta)$-joint eigenspace}
of $(A,B)$, and call a nonzero vector $v\in V_{\alpha,\beta}$
an \emph{$(\alpha,\beta)$-joint eigenvector} of $(A,B)$.
Suppose that $V_{\alpha,\beta}\neq 0$, and choose
$0\neq v\in V_{\alpha,\beta}$.
Since $AB=0$, we have
$$
0=ABv=\alpha\beta v.
$$ 
It follows that $\alpha\beta=0$, so $\alpha=0$ or $\beta=0$.
Hence every nonzero joint eigenspace of $(A,B)$ is of the form
$V_{\alpha,0}$ or $V_{0,\beta}$.

Let $V$ be a finite-dimensional $\widetilde{\mathcal H}$-module.
We say that $V$ is {\it $(A,B)$-multiplicity-free} whenever the following conditions hold:
\begin{itemize}
\item[(i)] $A$ and $B$ are diagonalizable on $V$;
\item[(ii)] every nonzero joint eigenspace of $(A,B)$ is one-dimensional.
\end{itemize}
Throughout this paper, we refer to an $(A,B)$-multiplicity-free module simply as a {\it multiplicity-free module}.

\begin{remark}\label{rmk:ds}
Let $V$ be a finite-dimensional multiplicity-free 
$\widetilde{\mathcal H}$-module.
Since $A$ and $B$ commute and are diagonalizable on $V$,
they are simultaneously diagonalizable.
Therefore
$$
V=\bigoplus_{\substack{(\alpha,\beta)\in\mathbb C^2\\
V_{\alpha,\beta}\neq 0}}
V_{\alpha,\beta},
$$
and each summand in this decomposition is one-dimensional.
\end{remark}

Let $V$ be a finite-dimensional $\widetilde{\mathcal H}$-module, and let
$W$ be a subspace of $V$.
For $i\in\{0,1\}$, we say that $t_i$ {\it preserves} $W$ if
$t_iW\subseteq W$.
Suppose that $W$ is two-dimensional, and let $\{u,v\}$ be an ordered basis of $W$.
If $t_i$ preserves $W$, then the restriction $t_i|_W$ is represented by a
$2\times2$ matrix with respect to $\{u,v\}$.
We call this matrix the {\it two-dimensional $t_i$-block} with respect to $\{u,v\}$.

We now introduce the notion of an adapted block basis, which plays a central role in our classification.

\begin{definition}\label{adapted basis}
Let $V$ be a finite-dimensional $\widetilde{\mathcal H}$-module.
\begin{enumerate}
\item[\rm (1)]
Suppose that $\dim(V)=2D$ for an integer $D\geq 1$.
An ordered basis
$$
b_0,a_1,b_1,\ldots,a_{D-1},b_{D-1},a_D
$$
is called an {\it even adapted block basis} if there exist nonzero scalars
$$
\lambda_1, \lambda_2, \dots , \lambda_D, \qquad
\mu_0, \mu_1, \dots, \mu_{D-1}
$$
such that the following conditions hold:
\begin{enumerate}
\item[(i)]
$a_i\in V_{\lambda_i,0}$ for $1\leq i\leq D$;
\item[(ii)]
$b_i\in V_{0,\mu_i}$ for $0\leq i\leq D-1$;
\item[(iii)]
$t_0$ preserves $\mathbb Cb_0$ and $\mathbb Ca_D$;
\item[(iv)]
$t_0$ preserves ${\rm span}\{a_i,b_i\}$ for $1\leq i\leq D-1$;
\item[(v)]
$t_1$ preserves ${\rm span}\{b_i,a_{i+1}\}$ for $0\leq i\leq D-1$.
\end{enumerate}

\item[\rm (2)]
Suppose that $\dim(V)=2D+1$ for an integer $D\geq 1$.
An ordered basis
$$
a_0,b_0,a_1,b_1,\ldots,a_{D-1},b_{D-1},a_D
$$
is called an {\it odd adapted block basis} if there exist nonzero scalars
$$
\lambda_0, \lambda_1, \dots , \lambda_D, \qquad
\mu_0, \mu_1, \dots, \mu_{D-1}
$$
such that the following conditions hold:
\begin{enumerate}
\item[(i)]
$a_i\in V_{\lambda_i,0}$ for $0\leq i\leq D$;
\item[(ii)]
$b_i\in V_{0,\mu_i}$ for $0\leq i\leq D-1$;
\item[(iii)]
$t_0$ preserves $\mathbb Ca_D$ and $t_1$ preserves $\mathbb Ca_0$;
\item[(iv)]
$t_0$ preserves ${\rm span}\{a_i,b_i\}$ for $0\leq i\leq D-1$;
\item[(v)]
$t_1$ preserves ${\rm span}\{b_i,a_{i+1}\}$ for $0\leq i\leq D-1$.
\end{enumerate}
\smallskip
\noindent
If $\dim(V)=1$, a basis $\{a_0\}$ is called an
{\it odd adapted block basis} if
$a_0\in V_{\lambda_0,0}$ for some $\lambda_0\in\mathbb C^\times$.
\end{enumerate}
When no confusion can arise, we simply call an even or odd adapted block basis an {\it adapted block basis}.
\end{definition}

\begin{remark}
If a finite-dimensional $\widetilde{\mathcal H}$-module $V$ has an adapted block basis, then
$V_{0,0}=0$.
Indeed, write $v\in V_{0,0}$ as a linear combination of the vectors in an adapted block basis.
Since all $\lambda_i$ and $\mu_i$ are nonzero, the equations
$Av=0$ and $Bv=0$ imply that the coefficients of the $a_i$ and the $b_i$, respectively, are zero.
Hence $v=0$.
\end{remark}

\begin{remark}
For a finite-dimensional $\widetilde{\mathcal H}$-module,
$(A,B)$-multiplicity-freeness and the existence of an adapted block basis
are independent conditions.
Multiplicity-freeness guarantees a basis of joint eigenvectors of $(A,B)$,
but it does not imply the preservation conditions for $t_0$ and $t_1$
in Definition~\ref{adapted basis}.
Conversely, a module with an adapted block basis need not be
multiplicity-free, since the definition does not require the $\lambda_i$
or the $\mu_i$ to be mutually distinct.
\end{remark}

We next consider the irreducibility of modules with an adapted block basis.
Recall that a module is \emph{irreducible} if it has no nonzero proper submodule.
We first give a criterion for determining when a subspace spanned by a
subset of an adapted block basis is an $\widetilde{\mathcal H}$-submodule.

\begin{lemma}\label{lem:submod-criterion}
Let $V$ be a finite-dimensional
$\widetilde{\mathcal H}$-module with an adapted block basis.
Let $U$ be a subspace of $V$ spanned by a subset of the adapted block
basis.
If $t_0$ and $t_1$ preserve $U$, then $U$ is an
$\widetilde{\mathcal H}$-submodule of $V$.
\end{lemma}

\begin{proof}
Since $t_0$ and $t_1$ are invertible and preserve the finite-dimensional
space $U$, their restrictions to $U$ are bijective.
Hence $t_0^{-1}$ and $t_1^{-1}$ preserve $U$.
Since $U$ is spanned by joint eigenvectors of $(A,B)$, both $A$ and $B$ preserve $U$.
It follows from
$$
u_0=At_0^{-1},\qquad u_1=t_1^{-1}B
$$
that $u_0$ and $u_1$ preserve $U$.
Thus $U$ is preserved by the generators
$t_0,u_0,t_1,u_1$ of $\widetilde{\mathcal H}$, and hence is an
$\widetilde{\mathcal H}$-submodule of $V$.
\end{proof}

The following example shows how a zero off-diagonal entry in a
two-dimensional $t_i$-block gives rise to a nonzero proper submodule.

\begin{example}\label{reducible example}
Let $V$ be an $\widetilde{\mathcal H}$-module with an adapted block basis
$\{b_0,a_1,b_1,a_2\}$.
By Definition~\ref{adapted basis}(1)(iv), $t_0$ preserves
${\rm span}\{a_1,b_1\}$.
Assume that the two-dimensional $t_0$-block with respect to $\{a_1,b_1\}$ is
\begin{equation*}
\begin{pmatrix}
p&0\\
r&s
\end{pmatrix}.
\end{equation*}
Consider the nonzero proper subspace $W={\rm span}\{b_1,a_2\}$ of $V$.
From the displayed matrix, $t_0b_1=sb_1$, and by Definition~\ref{adapted basis}(1)(iii), $t_0a_2 \in \mathbb{C}a_2$. 
Hence $t_0$ preserves $W$.
By Definition~\ref{adapted basis}(1)(v), $t_1$ also preserves $W$. 
It follows from Lemma~\ref{lem:submod-criterion} that $W$ is an $\widetilde{\mathcal H}$-submodule of $V$.
Therefore the module $V$ is reducible.
\end{example}

Example~\ref{reducible example} illustrates how
a zero off-diagonal entry in a two-dimensional $t_0$-block
gives rise to a nonzero proper $\widetilde{\mathcal H}$-submodule.
The following proposition generalizes this observation.

\begin{proposition}\label{reducible condition}
Let $V$ be a finite-dimensional $\widetilde{\mathcal H}$-module with an adapted block basis.
If, for some $i\in\{0,1\}$, a two-dimensional $t_i$-block has a zero off-diagonal entry, then $V$ is reducible.
\end{proposition}

\begin{proof}
If $\dim(V)=1$, then there are no two-dimensional $t_i$-blocks, and there
is nothing to prove.
Hence we may assume that $\dim(V)\geq 2$.

\medskip
\noindent
{\it Case 1.}
Suppose that $\dim(V)=2D$ with $D\geq 1$, and let
$$
\{b_0,a_1,b_1,\ldots,a_{D-1},b_{D-1},a_D\}
$$
be an adapted block basis of $V$.

We first consider the two-dimensional $t_0$-blocks.
If $D=1$, then there are no such blocks.
Thus for the $t_0$-blocks, we may assume that $D\geq 2$.
Let $1\leq i\leq D-1$, and write the $t_0$-block with respect to
$\{a_i,b_i\}$ as
$$
\begin{pmatrix}
p&q\\
r&s
\end{pmatrix}.
$$
Suppose that $q=0$, and set
$$
W={\rm span}\{b_i,a_{i+1},b_{i+1},\ldots,
a_{D-1},b_{D-1},a_D\}.
$$
Then $t_0b_i=sb_i$.
By Definition~\ref{adapted basis}(1)(iii) and (iv), $t_0$ preserves
${\rm span}\{a_j,b_j\}$ for $i+1\leq j\leq D-1$ and
$\mathbb Ca_D$.
Hence $t_0$ preserves $W$.
By Definition~\ref{adapted basis}(1)(v), $t_1$ preserves
${\rm span}\{b_j,a_{j+1}\}$ for $i\leq j\leq D-1$, and hence
preserves $W$.
It follows from Lemma~\ref{lem:submod-criterion} that $W$ is an
$\widetilde{\mathcal H}$-submodule of $V$.
Since $W$ is nonzero and proper, $V$ is reducible.

\noindent
Suppose that $r=0$, and set
$$
W'={\rm span}\{b_0,a_1,b_1,\ldots,
a_{i-1},b_{i-1},a_i\}.
$$
Then $t_0a_i=pa_i$.
By Definition~\ref{adapted basis}(1)(iii) and (iv), $t_0$ preserves
$\mathbb Cb_0$ and
${\rm span}\{a_j,b_j\}$ for $1\leq j\leq i-1$.
Hence $t_0$ preserves $W'$.
By Definition~\ref{adapted basis}(1)(v), $t_1$ preserves
${\rm span}\{b_j,a_{j+1}\}$ for $0\leq j\leq i-1$, and hence
preserves $W'$.
It follows from Lemma~\ref{lem:submod-criterion} that $W'$ is an
$\widetilde{\mathcal H}$-submodule of $V$.
Since $W'$ is nonzero and proper, $V$ is reducible.

We next consider the two-dimensional $t_1$-blocks.
Let $0\leq i\leq D-1$, and write the $t_1$-block with respect to
$\{b_i,a_{i+1}\}$ as
$
\begin{pmatrix}
p' & q'\\
r' & s'
\end{pmatrix}.
$
By an argument similar to that for the $t_0$-blocks, if $q'=0$, then
$$
{\rm span}\{a_{i+1},b_{i+1},\ldots,
a_{D-1},b_{D-1},a_D\}
$$
is a nonzero proper $\widetilde{\mathcal H}$-submodule of $V$, and if
$r'=0$, then
$$
{\rm span}\{b_0,a_1,b_1,\ldots,a_i,b_i\}
$$
is a nonzero proper $\widetilde{\mathcal H}$-submodule of $V$.
Hence $V$ is reducible in either case.

\medskip
\noindent
{\it Case 2.}
Suppose that $\dim(V)=2D+1$ with $D\geq 1$.
The argument is similar to that in Case~1.

This completes the proof.
\end{proof}

Proposition~\ref{reducible condition} gives the following corollary.

\begin{corollary}\label{irreducible condition}
Let $V$ be a finite-dimensional $\widetilde{\mathcal H}$-module with an
adapted block basis.
If $V$ is irreducible, then, for each $i\in\{0,1\}$, both off-diagonal
entries of every two-dimensional $t_i$-block are nonzero.
\end{corollary}

\section{The $\widetilde{\mathcal{H}}$-modules $E_D$ and $O_D$}\label{Sec:ED,OD}

In this section, we construct two families of finite-dimensional 
$\widetilde{\mathcal{H}}$-modules and give sufficient conditions for their 
irreducibility. 
We begin with some notation.
For $c,r,s,\lambda,\mu \in \mathbb{C}^{\times}$, set $k(r)=r+r^{-1}$ and define
\begin{equation}
R(r,\lambda,c)=
\begin{pmatrix}
k(r)-\lambda^{-1}&-c\lambda^{-1}\\
c^{-1}(\lambda-k(r)+\lambda^{-1})&\lambda^{-1}
\end{pmatrix},
\qquad
P(c)=
\begin{pmatrix}
1&c\\
0&0
\end{pmatrix},
\label{local block R}
\end{equation}
\begin{equation}
S(r,s)=
\begin{pmatrix}
k(r)&s\\
-s^{-1}&0
\end{pmatrix},
\qquad
N(s,\mu)=
\begin{pmatrix}
0&0\\
s^{-1}\mu&0
\end{pmatrix}.
\label{local blocks PSN}
\end{equation}

\begin{lemma}\label{local block lemma}
Let $c,r,s,\lambda,\mu \in \mathbb{C}^{\times}$.
The following assertions hold.
\begin{enumerate}
\item[\rm(i)]
$P(c)^2=P(c)$ and $N(s,\mu)^2=0$.
\item[\rm(ii)]
$\operatorname{tr} R(r,\lambda,c)=k(r)$ and
$\det(R(r,\lambda,c))=1$.
\item[\rm(iii)]
$\operatorname{tr} S(r,s)=k(r)$ and
$\det(S(r,s))=1$.
\item[\rm(iv)]
$(R(r,\lambda,c)-rI)(R(r,\lambda,c)-r^{-1}I)=0$ and
$(S(r,s)-rI)(S(r,s)-r^{-1}I)=0$.
\item[\rm(v)]
$P(c)R(r,\lambda,c)=
\begin{pmatrix}
\lambda&0\\
0&0
\end{pmatrix}$.
\item[\rm(vi)]
$S(r,s)N(s,\mu)=
\begin{pmatrix}
\mu&0\\
0&0
\end{pmatrix}$.
\end{enumerate}
\end{lemma}

\begin{proof}
Assertions (i)--(iii), (v), and (vi) follow by direct computation.
By (ii), (iii), and the Cayley-Hamilton theorem, each matrix
$X\in\{R(r,\lambda,c),S(r,s)\}$
satisfies
$$
X^2-k(r)X+I=(X-rI)(X-r^{-1}I)=0.
$$
This proves (iv).
\end{proof}

\subsection{The even-dimensional $\widetilde{\mathcal{H}}$-modules $E_D$}\label{even subsec}

In this subsection, we construct a family of even-dimensional
$\widetilde{\mathcal{H}}$-modules and give sufficient conditions for their
irreducibility. 
Throughout this subsection, let $D\geq 1$.
Recall the nonzero scalars $r_0,r_1$, and choose
\begin{equation}
\rho_-,\rho_+ \in \{r_0,r_0^{-1}\}.
\label{even rho}
\end{equation}
Choose the following tuples of nonzero scalars
\begin{align}
&&  c & = (c_1,\dots,c_{D-1}), & s &= (s_0,\dots,s_{D-1}), && \label{even cs 1}\\
&&  \lambda & = (\lambda_1,\dots,\lambda_{D-1}), & \mu &= (\mu_0,\dots,\mu_{D-1}). &&\label{even lambda mu 2}
\end{align}
When $D=1$, the tuples $c$ and $\lambda$ are understood to be empty.
Let
$$
E_D:=E_D(\rho_-,\rho_+;c,s,\lambda,\mu)
$$
be a $2D$-dimensional vector space with an ordered basis
$$
\mathfrak{B}_E := \{b_0,a_1,b_1,\dots,a_{D-1},b_{D-1},a_D\}.
$$
Define matrices $T_0,U_0,T_1,U_1 \in \operatorname{Mat}_{2D}(\mathbb{C})$ by
\begin{align}
T_0&=\operatorname{blockdiag}
([\rho_-],R_1,\dots,R_{D-1},[\rho_+]),
\label{e.mat:T0}\\
U_0&=\operatorname{blockdiag}
([0],P_1,\dots,P_{D-1},[1]),
\label{e.mat:U0}\\
T_1&=\operatorname{blockdiag}
(S_0,S_1,\dots,S_{D-1}),
\label{e.mat:T1}\\
U_1&=\operatorname{blockdiag}
(N_0,N_1,\dots,N_{D-1}),
\label{e.mat:U1}
\end{align}
where
$$
R_i=R(r_0,\lambda_i,c_i),\qquad
P_i=P(c_i)
\qquad (1\leq i\leq D-1),
$$
and
$$
S_i=S(r_1,s_i),\qquad
N_i=N(s_i,\mu_i)
\qquad (0\leq i\leq D-1).
$$
When $D=1$, there are no blocks $R_i$ or $P_i$, and the above definitions reduce to
$$
T_0=\operatorname{diag}(\rho_-,\rho_+),\qquad
U_0=\operatorname{diag}(0,1),\qquad
T_1=S(r_1,s_0),\qquad
U_1=N(s_0,\mu_0).
$$

\begin{lemma}\label{nil-DAHA condition}
The matrices $T_0,U_0,T_1,U_1$ satisfy the following relations.
\begin{enumerate}
\item[\rm(i)] $(T_i-r_iI)(T_i-r_i^{-1}I)=0$ for $i\in\{0,1\}$.
\item[\rm(ii)] $U_0^2=U_0$.
\item[\rm(iii)] $U_1^2=0$.
\item[\rm(iv)] $U_0T_0T_1U_1=0=T_1U_1U_0T_0$.
\end{enumerate}
\end{lemma}

\begin{proof}
Assertions (i)--(iii) follow from Lemma~\ref{local block lemma} and \eqref{e.mat:T0}--\eqref{e.mat:U1}.
By Lemma~\ref{local block lemma}(v) and (vi), the matrices $U_0T_0$ and
$T_1U_1$ are diagonal, and their nonzero diagonal entries occur in
disjoint positions. 
Thus
$(U_0T_0)(T_1U_1)=0=(T_1U_1)(U_0T_0)$, 
which proves (iv).
\end{proof}

\begin{proposition}\label{prop:E_D}
There exists an $\widetilde{\mathcal{H}}$-module structure on $E_D$ such that
the matrices representing $t_0,u_0,t_1,u_1$ with respect to
$\mathfrak{B}_E$ are $T_0,U_0,T_1,U_1$, respectively.
\end{proposition}

\begin{proof}
By Lemma~\ref{nil-DAHA condition}, there is a unique
$\mathbb{C}$-algebra homomorphism $\widetilde{\mathcal{H}}\rightarrow
\operatorname{Mat}_{2D}(\mathbb{C})$
such that
$$
t_0 \mapsto T_0,\qquad
u_0  \mapsto U_0,\qquad
t_1  \mapsto T_1,\qquad
u_1 \mapsto U_1.
$$
Using the basis $\mathfrak{B}_E$ to identify
$\operatorname{Mat}_{2D}(\mathbb{C})$ with $\operatorname{End}(E_D)$,
this homomorphism endows $E_D$ with a
$\widetilde{\mathcal{H}}$-module structure.
\end{proof}

Recall the elements $A$ and $B$ of $\widetilde{\mathcal{H}}$ from \eqref{eq:A,B}.
The matrices representing $A$ and $B$ with respect to
$\mathfrak{B}_E$ are $U_0T_0$ and $T_1U_1$, respectively. 
For $D\geq 2$, Lemma~\ref{local block lemma}(v) and (vi) give
\begin{align}
U_0T_0 & = \operatorname{diag}
(0,\lambda_1,0,\lambda_2,0,\dots,\lambda_{D-1},0,\rho_+), \label{eq:mat.AB(1)}\\
T_1U_1 & = \operatorname{diag} \label{eq:mat.AB(2)}
(\mu_0,0,\mu_1,0,\dots,\mu_{D-1},0).
\end{align}
When $D=1$, these matrices reduce to
\begin{equation}\label{eq:mat.AB(3)}
U_0T_0=\operatorname{diag}(0,\rho_+),
\qquad
T_1U_1=\operatorname{diag}(\mu_0,0).
\end{equation}

\begin{proposition}\label{prop:E_D,abb}
The basis $\mathfrak{B}_E$ is an adapted block basis for $E_D$.
\end{proposition}

\begin{proof}
This follows from Definition~\ref{adapted basis},
\eqref{e.mat:T0}, \eqref{e.mat:T1}, and
\eqref{eq:mat.AB(1)}--\eqref{eq:mat.AB(3)}.
\end{proof}

The following theorem gives sufficient conditions for $E_D$ to be
multiplicity-free and irreducible.

\begin{theorem}\label{even irreducible}
Assume that $\lambda_1,\dots,\lambda_{D-1}$ are mutually distinct and
$\lambda_i\notin\{r_0,r_0^{-1}\}$ for $1\leq i\leq D-1$, and that
$\mu_0,\dots,\mu_{D-1}$ are mutually distinct.
Then the $\widetilde{\mathcal{H}}$-module $E_D$ is multiplicity-free and irreducible.
\end{theorem}

\begin{proof}
Let $W$ be a nonzero $\widetilde{\mathcal{H}}$-submodule of $E_D$.
By the given conditions, the module $E_D$ is
$(A,B)$-multiplicity-free.
Since $W$ is invariant under $A$ and $B$,
Remark~\ref{rmk:ds} and
Proposition~\ref{prop:E_D,abb} imply that $W$ is spanned by a subset of
$\mathfrak{B}_E$.
In particular, $W$ contains $a_{i+1}$ or $b_i$ for some
$0\leq i\leq D-1$.

Suppose first that $b_i\in W$ for some $0\leq i\leq D-1$.
Since
$$
t_1b_i=k(r_1)b_i-s_i^{-1}a_{i+1}\in W,
$$
we have $a_{i+1}\in W$.
If $i\leq D-2$, then
$$
t_0a_{i+1}
=
\bigl(k(r_0)-\lambda_{i+1}^{-1}\bigr)a_{i+1}
+
c_{i+1}^{-1}
\bigl(\lambda_{i+1}-k(r_0)+\lambda_{i+1}^{-1}\bigr)b_{i+1}
\in W.
$$
The coefficient of $b_{i+1}$ is nonzero, since $c_{i+1}\neq0$ and
$$
\lambda_{i+1}-k(r_0)+\lambda_{i+1}^{-1}
=
\lambda_{i+1}^{-1}
(\lambda_{i+1}-r_0)(\lambda_{i+1}-r_0^{-1})
\neq0.
$$
Thus $b_{i+1}\in W$.
Repeating these steps as necessary gives
\begin{equation}\label{pf:bj,aj+1(1)}
b_j,a_{j+1}\in W
\qquad (i\leq j\leq D-1).
\end{equation}
If $i\geq1$, then
$t_0b_i = -c_i\lambda_i^{-1}a_i+\lambda_i^{-1}b_i \in W$.
Since $c_i\lambda_i^{-1}\neq0$, we have $a_i\in W$.
It follows that
$$
t_1a_i=s_{i-1}b_{i-1}\in W,
$$
and hence $b_{i-1}\in W$.
Repeating these steps as necessary gives
\begin{equation}\label{pf:bj,aj+1(2)}
b_j,a_{j+1}\in W
\qquad (0\leq j\leq i-1).
\end{equation}
By \eqref{pf:bj,aj+1(1)} and \eqref{pf:bj,aj+1(2)}, every vector in $\mathfrak{B}_E$ belongs to $W$, and therefore
$W=E_D$.

Next, suppose that $a_{i+1}\in W$ for some $0\leq i\leq D-1$.
Since $t_1a_{i+1}=s_ib_i\in W$, we have $b_i\in W$.
Thus $W=E_D$ by the preceding argument.
Consequently, $E_D$ is irreducible.
\end{proof}

\subsection{The odd-dimensional $\widetilde{\mathcal{H}}$-modules $O_D$}\label{subsec:O_D}

In this subsection, we construct a family of odd-dimensional
$\widetilde{\mathcal{H}}$-modules and give sufficient conditions for their
irreducibility.
Throughout this subsection, let $D\geq 0$.
Recall the nonzero scalars $r_0,r_1$, and choose
\begin{equation}
\rho \in \{r_0,r_0^{-1}\},\qquad
\sigma \in \{r_1,r_1^{-1}\}.
\label{odd rho}
\end{equation}
Choose the following tuples of nonzero scalars
\begin{align}
&& c &=(c_0,\dots,c_{D-1}), & s &=(s_0,\dots,s_{D-1}), && \label{odd cs}\\
&& \lambda &=(\lambda_0,\dots,\lambda_{D-1}), &
\mu &=(\mu_0,\dots,\mu_{D-1}). && \label{odd lambda mu}
\end{align}
When $D=0$, the tuples $c,s,\lambda,\mu$ are understood to be empty.
Let
$$
O_D:=O_D(\rho,\sigma;c,s,\lambda,\mu)
$$
be a $(2D+1)$-dimensional vector space.
For $D\geq 1$, let
$$
\mathfrak{B}_O
:=\{a_0,b_0,a_1,b_1,\dots,a_{D-1},b_{D-1},a_D\}
$$
be an ordered basis of $O_D$, and for $D=0$, let $\mathfrak{B}_O := \{a_0\}$.
Define matrices
$T_0,U_0,T_1,U_1\in\operatorname{Mat}_{2D+1}(\mathbb{C})$ by
\begin{align}
T_0&=\operatorname{blockdiag}
(R_0,R_1,\dots,R_{D-1},[\rho]),
\label{o.mat:T0}\\
U_0&=\operatorname{blockdiag}
(P_0,P_1,\dots,P_{D-1},[1]),
\label{o.mat:U0}\\
T_1&=\operatorname{blockdiag}
([\sigma],S_0,S_1,\dots,S_{D-1}),
\label{o.mat:T1}\\
U_1&=\operatorname{blockdiag}
([0],N_0,N_1,\dots,N_{D-1}),
\label{o.mat:U1}
\end{align}
where
$$
R_i=R(r_0,\lambda_i,c_i),\qquad
P_i=P(c_i)
\qquad (0\leq i\leq D-1),
$$
and
$$
S_i=S(r_1,s_i),\qquad
N_i=N(s_i,\mu_i)
\qquad (0\leq i\leq D-1).
$$
When $D=0$, there are no blocks $R_i$, $P_i$, $S_i$, or $N_i$, and the above definitions reduce to
$$
T_0=[\rho],\qquad
U_0=[1],\qquad
T_1=[\sigma],\qquad
U_1=[0].
$$

\begin{proposition}\label{prop:O_D}
There exists an $\widetilde{\mathcal{H}}$-module structure on $O_D$ such that
the matrices representing $t_0,u_0,t_1,u_1$ with respect to
$\mathfrak{B}_O$ are $T_0,U_0,T_1,U_1$, respectively.
\end{proposition}

\begin{proof}
For $D=0$, the result follows immediately.
For $D \geq 1$, the result follows by the same arguments as in the proofs of
Lemma~\ref{nil-DAHA condition} and Proposition~\ref{prop:E_D}, using
\eqref{o.mat:T0}--\eqref{o.mat:U1}.
\end{proof}

For $D\geq 1$, the matrices representing $A$ and $B$ with respect to
$\mathfrak{B}_O$ are given as follows:
\begin{align*}
A: \quad &\operatorname{diag}
(\lambda_0,0,\lambda_1,0,\dots,\lambda_{D-1},0,\rho),\\
B: \quad &\operatorname{diag}
(0,\mu_0,0,\mu_1,\dots,0,\mu_{D-1},0).
\end{align*}
When $D=0$, these matrices are $[\rho]$ and $[0]$, respectively.

\begin{proposition}\label{prop:O_D,abb}
The basis $\mathfrak{B}_O$ is an adapted block basis for $O_D$.
\end{proposition}

\begin{proof}
Similar to the proof of Proposition~\ref{prop:E_D,abb}.
\end{proof}

The following theorem gives sufficient conditions for $O_D$ to be
multiplicity-free and irreducible.

\begin{theorem}\label{odd irreducible}
Assume that $\lambda_0,\dots,\lambda_{D-1}$ are mutually
distinct and $\lambda_i\notin\{r_0,r_0^{-1}\}$ for $0\leq i\leq D-1$, and that
$\mu_0,\dots,\mu_{D-1}$ are mutually distinct.
Then the $\widetilde{\mathcal{H}}$-module $O_D$ is multiplicity-free and irreducible.
\end{theorem}

\begin{proof}
For $D=0$, the result follows immediately.
For $D\geq 1$, the proof is similar to that of
Theorem~\ref{even irreducible}.
\end{proof}

\section{Main Results}\label{Sec:MR}

In this section, we classify the multiplicity-free finite-dimensional irreducible $\widetilde{\mathcal{H}}$-modules with an adapted block basis.
To this end, we establish preliminary lemmas.

We begin with the even-dimensional case.
Recall the parameters $\rho_-,\rho_+$ and the tuples
$c,s,\lambda,\mu$ from
\eqref{even rho}--\eqref{even lambda mu 2}.
In the previous section, we constructed the
$\widetilde{\mathcal{H}}$-module
$E_D(\rho_-,\rho_+;c,s,\lambda,\mu)$.
For any of the tuples $c,s,\lambda,\mu$, we use ${\bf 1}$
to denote the tuple of the same length whose entries are all $1$.
The following lemma shows that the tuples $c$ and $s$ do not affect 
the isomorphism class of the module.

\begin{lemma}\label{even parameter}
For $D\geq 1$, the following assertions hold.
\begin{enumerate}
\item[\rm(i)]
The $\widetilde{\mathcal{H}}$-modules
$E_D(\rho_-,\rho_+;c,s,\lambda,\mu)$ and
$E_D(\rho_-,\rho_+;c',s,\lambda,\mu)$
are isomorphic for any tuple $c'$ of nonzero scalars.

\item[\rm(ii)]
The $\widetilde{\mathcal{H}}$-modules
$E_D(\rho_-,\rho_+;c,s,\lambda,\mu)$ and
$E_D(\rho_-,\rho_+;c,s',\lambda,\mu)$
are isomorphic for any tuple $s'$ of nonzero scalars.
\end{enumerate}
\end{lemma}

\begin{proof}
(i)
It suffices to show that, for any tuple
$c=(c_1,\dots,c_{D-1})$ of nonzero scalars, the
$\widetilde{\mathcal{H}}$-modules
$E_D(\rho_-,\rho_+;{\bf 1},s,\lambda,\mu)$ and
$E_D(\rho_-,\rho_+;c,s,\lambda,\mu)$
are isomorphic.
We abbreviate
$$
V=E_D(\rho_-,\rho_+;{\bf 1},s,\lambda,\mu),
\qquad
V'=E_D(\rho_-,\rho_+;c,s,\lambda,\mu).
$$
Let
$$
\{b_0,a_1,b_1,\dots,a_{D-1},b_{D-1},a_D\}
\quad \text{and} \quad
\{b_0',a_1',b_1',\dots,a_{D-1}',b_{D-1}',a_D'\}
$$
be the adapted block bases of $V$ and $V'$, respectively.
Set
$$
\ell_0=1,
\qquad
\ell_i=c_1^{-1}\cdots c_i^{-1}
\qquad (1\leq i\leq D-1).
$$
Define a linear map
$\varphi:V\to V'$
by
$$
a_i\mapsto \ell_{i-1}a_i'
\quad (1\leq i\leq D),
\qquad
b_i\mapsto \ell_i b_i'
\quad (0\leq i\leq D-1).
$$
Observe that $\varphi$ is a $\mathbb{C}$-vector space isomorphism
since each $\ell_i$ is nonzero.
We claim that $\varphi$ commutes with the actions of
$t_0,u_0,t_1,u_1$.

For $1\leq i\leq D-1$, consider the restriction of $\varphi$ to
$\operatorname{span}\{a_i,b_i\}$.
With respect to the ordered bases $\{a_i,b_i\}$ and $\{a_i',b_i'\}$,
the matrix representing this restriction is
$$
L_i=
\begin{pmatrix}
\ell_{i-1}&0\\
0&\ell_i
\end{pmatrix}.
$$
Note that the matrices representing the $t_0$- and $u_0$-actions on
$\operatorname{span}\{a_i,b_i\}$ in $V$ are
$R(r_0,\lambda_i,1)$ and $P(1)$, and the corresponding matrices on
$\operatorname{span}\{a_i',b_i'\}$ in $V'$ are
$R(r_0,\lambda_i,c_i)$ and $P(c_i)$.
Since $\ell_i=c_i^{-1}\ell_{i-1}$, a routine computation gives
$$
L_iR(r_0,\lambda_i,1)=R(r_0,\lambda_i,c_i)L_i,
\qquad
L_iP(1)=P(c_i)L_i.
$$
Therefore
$$
\varphi(t_0v)=t_0\varphi(v),
\qquad
\varphi(u_0v)=u_0\varphi(v)
$$
for every $v\in\operatorname{span}\{a_i,b_i\}$.
Moreover, on each of $\mathbb{C}b_0$ and $\mathbb{C}b_0'$, the matrices
representing the $t_0$- and $u_0$-actions are $[\rho_-]$ and $[0]$,
respectively.
On each of $\mathbb{C}a_D$ and $\mathbb{C}a_D'$, the matrices
representing the $t_0$- and $u_0$-actions are $[\rho_+]$ and $[1]$, 
respectively.
Thus $\varphi$ also commutes with the $t_0$- and $u_0$-actions on
$\mathbb{C}b_0$ and $\mathbb{C}a_D$.

Next, for $0\leq i\leq D-1$, consider the restriction of $\varphi$
to $\operatorname{span}\{b_i,a_{i+1}\}$.
With respect to the ordered bases $\{b_i,a_{i+1}\}$ and
$\{b_i',a_{i+1}'\}$, the matrix representing this restriction is
$\ell_iI$.
Note that the matrices representing the $t_1$- and $u_1$-actions on
$\operatorname{span}\{b_i,a_{i+1}\}$ in $V$ are
$S(r_1,s_i)$ and $N(s_i,\mu_i)$, and the corresponding matrices on
$\operatorname{span}\{b_i',a_{i+1}'\}$ in $V'$ are the same.
Since $\ell_iI$ commutes with both $S(r_1,s_i)$ and $N(s_i,\mu_i)$, we have
$$
\varphi(t_1v)=t_1\varphi(v),
\qquad
\varphi(u_1v)=u_1\varphi(v)
$$
for every $v\in\operatorname{span}\{b_i,a_{i+1}\}$.

Together with the above observations, the claim follows.
Thus $\varphi$ is an
$\widetilde{\mathcal{H}}$-module isomorphism.
This proves (i).

(ii)
It suffices to show that, for any tuple
$s=(s_0,\dots,s_{D-1})$ of nonzero scalars, the
$\widetilde{\mathcal{H}}$-modules
$E_D(\rho_-,\rho_+;c,{\bf 1},\lambda,\mu)$ and
$E_D(\rho_-,\rho_+;c,s,\lambda,\mu)$
are isomorphic.
We use an argument similar to that in (i).
Abbreviate
$$
W=E_D(\rho_-,\rho_+;c,{\bf 1},\lambda,\mu),
\qquad
W'=E_D(\rho_-,\rho_+;c,s,\lambda,\mu).
$$
Let
$$
\{b_0,a_1,b_1,\dots,a_{D-1},b_{D-1},a_D\}
\quad \text{and} \quad
\{b_0',a_1',b_1',\dots,a_{D-1}',b_{D-1}',a_D'\}
$$
be the adapted block bases of $W$ and $W'$, respectively.
Set
$$
k_0=1,
\qquad
k_i=s_0^{-1}\cdots s_{i-1}^{-1}
\qquad (1\leq i\leq D).
$$
Define a linear map
$\psi:W\to W'$
by
$$
a_i\mapsto k_i a_i'
\quad (1\leq i\leq D),
\qquad
b_i\mapsto k_i b_i'
\quad (0\leq i\leq D-1).
$$
Observe that $\psi$ is a $\mathbb{C}$-vector space isomorphism
since each $k_i$ is nonzero.
It remains to show that $\psi$ commutes with the actions of
$t_0,u_0,t_1,u_1$.

For $1\leq i\leq D-1$, the restriction of $\psi$ to
$\operatorname{span}\{a_i,b_i\}$ has representing matrix $k_iI$,
and the corresponding $t_0$- and $u_0$-blocks in both $W$ and $W'$
are $R(r_0,\lambda_i,c_i)$ and $P(c_i)$, respectively.
Thus $\psi$ commutes with the $t_0$- and $u_0$-actions on
$\operatorname{span}\{a_i,b_i\}$.
The same argument applies to $\mathbb{C}b_0$ and
$\mathbb{C}a_D$, so $\psi$ commutes with the $t_0$- and $u_0$-actions.

For $0\leq i\leq D-1$, since $k_{i+1}=k_is_i^{-1}$, the restriction
of $\psi$ to $\operatorname{span}\{b_i,a_{i+1}\}$ has representing
matrix $k_iQ_i$, where
$$
Q_i=
\begin{pmatrix}
1&0\\
0&s_i^{-1}
\end{pmatrix}.
$$
The corresponding $t_1$- and $u_1$-blocks in $W$ are
$S(r_1,1)$ and $N(1,\mu_i)$, and those in $W'$ are
$S(r_1,s_i)$ and $N(s_i,\mu_i)$.
A routine computation gives
$$
Q_iS(r_1,1)=S(r_1,s_i)Q_i,
\qquad
Q_iN(1,\mu_i)=N(s_i,\mu_i)Q_i.
$$
Thus $\psi$ commutes with the $t_1$- and $u_1$-actions.
Consequently, $\psi$ is an
$\widetilde{\mathcal{H}}$-module isomorphism.
This proves (ii).
\end{proof}

Next, we consider the odd-dimensional case.
Recall the parameters $\rho,\sigma$ and the tuples
$c,s,\lambda,\mu$ from
\eqref{odd rho}--\eqref{odd lambda mu}.
In Subsection~\ref{subsec:O_D}, we constructed the
$\widetilde{\mathcal{H}}$-module
$O_D(\rho,\sigma;c,s,\lambda,\mu)$.
The following lemma shows that the tuples $c$ and $s$ do not affect
the isomorphism class of the module.

\begin{lemma}\label{odd parameter}
For $D\geq 0$, the following assertions hold.
\begin{enumerate}
\item[\rm(i)]
The $\widetilde{\mathcal{H}}$-modules
$O_D(\rho,\sigma;c,s,\lambda,\mu)$ and
$O_D(\rho,\sigma;c',s,\lambda,\mu)$
are isomorphic for any tuple $c'$ of nonzero scalars.
\item[\rm(ii)]
The $\widetilde{\mathcal{H}}$-modules
$O_D(\rho,\sigma;c,s,\lambda,\mu)$ and
$O_D(\rho,\sigma;c,s',\lambda,\mu)$
are isomorphic for any tuple $s'$ of nonzero scalars.
\end{enumerate}
\end{lemma}

\begin{proof}
It is similar to Lemma~\ref{even parameter}.
\end{proof}

We now define the parameter sets that will be used to state our main result.
For $D\geq 1$, let $\mathcal{E}_D$ be the set of tuples of parameters
\begin{center}
$(\rho_-,\rho_+;\lambda,\mu)$
\end{center}
satisfying the following conditions:
\begin{itemize}
\item the parameters $\rho_-,\rho_+$ and the tuples $\lambda,\mu$
are as in \eqref{even rho} and \eqref{even lambda mu 2}, respectively;
\item $\lambda_1,\dots,\lambda_{D-1}
\in \mathbb{C}^{\times}\setminus\{r_0,r_0^{-1}\}$
are mutually distinct; 
\item $\mu_0,\mu_1,\dots,\mu_{D-1}\in\mathbb{C}^{\times}$
are mutually distinct.
\end{itemize}

Let $\mathcal{O}_0$ be the set of pairs $(\rho,\sigma)$,
where $\rho,\sigma$ are as in \eqref{odd rho}.
For $D\geq 1$, let $\mathcal{O}_D$ be the set of tuples of parameters
\begin{center}
$(\rho,\sigma;\lambda,\mu)$
\end{center}
satisfying the following conditions:
\begin{itemize}
\item the parameters $\rho,\sigma$ and the tuples $\lambda,\mu$
are as in \eqref{odd rho} and \eqref{odd lambda mu}, respectively;
\item $\lambda_0,\lambda_1,\dots,\lambda_{D-1}
\in \mathbb{C}^{\times}\setminus\{r_0,r_0^{-1}\}$
are mutually distinct; 
\item $\mu_0,\mu_1,\dots,\mu_{D-1}\in\mathbb{C}^{\times}$
are mutually distinct.
\end{itemize}

The following theorem is the main result of the paper.

\begin{theorem}\label{thm:main}
The following assertions hold.
\begin{enumerate}
\item[\rm (i)]
For $D\geq 1$, there is a one-to-one correspondence between the elements of
$\mathcal{E}_D$ and the isomorphism classes of the multiplicity-free
$2D$-dimensional irreducible $\widetilde{\mathcal{H}}$-modules
with an adapted block basis.

\item[\rm (ii)]
For $D\geq 0$, there is a one-to-one correspondence between the elements of
$\mathcal{O}_D$ and the isomorphism classes of the multiplicity-free
$(2D+1)$-dimensional irreducible $\widetilde{\mathcal{H}}$-modules
with an adapted block basis.
\end{enumerate}
\end{theorem}

\begin{proof}
(i)
Let $\mathfrak{A}$ be the set of isomorphism classes of the
multiplicity-free $2D$-dimensional irreducible
$\widetilde{\mathcal{H}}$-modules with an adapted block basis.
Define $\varphi:\mathcal{E}_D \to \mathfrak{A}$ by
$$
(\rho_-,\rho_+;\lambda,\mu)
\mapsto
[(\rho_-,\rho_+;\lambda,\mu)],
$$
where $[(\rho_-,\rho_+;\lambda,\mu)]$ is the isomorphism class containing
$E_D(\rho_-,\rho_+;{\bf 1},{\bf 1},\lambda,\mu)$.
By Proposition~\ref{prop:E_D,abb} and
Theorem~\ref{even irreducible}, the map $\varphi$ is well-defined.
We show that $\varphi$ is bijective.

First, we show that $\varphi$ is surjective.
Let $\mathcal{V}\in\mathfrak{A}$ and choose
$V\in\mathcal{V}$ with an adapted block basis
$\{b_0,a_1,b_1,\dots,a_{D-1},b_{D-1},a_D\}$.
For each $0\leq i\leq D-1$, there exist
$\lambda_{i+1},\mu_i\in\mathbb{C}^{\times}$ such that
$a_{i+1}\in V_{\lambda_{i+1},0}$ and $b_i\in V_{0,\mu_i}$.
We now determine the actions of $t_0$ and $u_0$ on the basis vectors $b_0$ and $a_D$.
Since $t_0$ preserves $\mathbb{C}b_0$, there exists
$\rho_-\in\mathbb{C}^{\times}$ such that
$t_0b_0=\rho_-b_0$.
By Definition~\ref{nil.H}(i), $(t_0-r_0)(t_0-r_0^{-1})=0$, and hence
$\rho_-\in\{r_0,r_0^{-1}\}$.
Similarly, there exists $\rho_+\in\{r_0,r_0^{-1}\}$ such that
$t_0a_D=\rho_+a_D$.
Observe that
$$
\lambda_Da_D=Aa_D=u_0t_0a_D=\rho_+u_0a_D.
$$
Applying $u_0$ to this equation and using $u_0^2=u_0$, we obtain
$\lambda_Du_0a_D=\lambda_Da_D$.
Since $\lambda_D\neq0$, it follows that $u_0a_D=a_D$, and hence
$\lambda_D=\rho_+$.
Also, since $Ab_0=0$ and $t_0b_0=\rho_-b_0$, we obtain
$u_0b_0=0$.

For $1\leq i\leq D-1$, let
$$
T_{0,i}=
\begin{pmatrix}
x_i&y_i\\
z_i&w_i
\end{pmatrix}
$$
be the two-dimensional $t_0$-block with respect to
$\{a_i,b_i\}$.
Since $V$ is irreducible, both $y_i$ and $z_i$ are nonzero by
Corollary~\ref{irreducible condition}.
Observe that the characteristic polynomial of $T_{0,i}$ is
$X^2-k(r_0)X+1$, so
\begin{equation}\label{T0,i:tr,det}
\operatorname{tr}(T_{0,i})=k(r_0),
\qquad
\det(T_{0,i})=1.
\end{equation}
Note that the two-dimensional $A$-block with respect to
$\{a_i,b_i\}$ is
$$
A_i=
\begin{pmatrix}
\lambda_i&0\\
0&0
\end{pmatrix}.
$$
Let $U_{0,i}$ be the two-dimensional $u_0$-block with respect to
$\{a_i,b_i\}$.
Then
$$
U_{0,i}
=(U_{0,i}T_{0,i})T_{0,i}^{-1}
=A_iT_{0,i}^{-1}
=
\begin{pmatrix}
\lambda_iw_i&-\lambda_i y_i\\
0&0
\end{pmatrix}.
$$
Since $U_{0,i}^2=U_{0,i}$ and $y_i\neq0$, we have
$\lambda_iw_i=1$, and hence
$w_i=\lambda_i^{-1}$.
Using this and \eqref{T0,i:tr,det}, we obtain $x_i=k(r_0)-\lambda_i^{-1}$
and
\begin{equation}
y_iz_i=-\lambda_i^{-1}
(\lambda_i-k(r_0)+\lambda_i^{-1})
= -\lambda_i^{-2}(\lambda_i-r_0)(\lambda_i-r_0^{-1}).
\label{lambda neq r_0}
\end{equation}
Set $c_i=-\lambda_i y_i$, which is nonzero.
It follows that
$$
T_{0,i}=R(r_0,\lambda_i,c_i),
\qquad
U_{0,i}=P(c_i),
$$
where $R(r_0,\lambda_i,c_i)$ and $P(c_i)$ are in
\eqref{local block R}.

For $0\leq i\leq D-1$, let
$$
T_{1,i}=
\begin{pmatrix}
\tilde{x}_i&\tilde{y}_i\\
\tilde{z}_i&\tilde{w}_i
\end{pmatrix}
$$
be the two-dimensional $t_1$-block with respect to
$\{b_i,a_{i+1}\}$.
Since $V$ is irreducible, both $\tilde{y}_i$ and $\tilde{z}_i$ are nonzero
by Corollary~\ref{irreducible condition}.
The characteristic polynomial of $T_{1,i}$ is
$X^2-k(r_1)X+1$, so
\begin{equation}\label{T1,i:tr,det}
\operatorname{tr}(T_{1,i})=k(r_1),
\qquad
\det(T_{1,i})=1.
\end{equation}
Note that the two-dimensional $B$-block with respect to
$\{b_i,a_{i+1}\}$ is
$$
B_i=
\begin{pmatrix}
\mu_i&0\\
0&0
\end{pmatrix}.
$$
Let $U_{1,i}$ be the two-dimensional $u_1$-block with respect to
$\{b_i,a_{i+1}\}$.
Then
$$
U_{1,i}
=T_{1,i}^{-1}(T_{1,i}U_{1,i})
=T_{1,i}^{-1}B_i
=
\begin{pmatrix}
\tilde{w}_i\mu_i&0\\
-\tilde{z}_i\mu_i&0
\end{pmatrix}.
$$
Since $U_{1,i}^2=0$ and $\mu_i\neq0$, we have
$\tilde{w}_i=0$.
Using this and \eqref{T1,i:tr,det}, we obtain
$\tilde{x}_i=k(r_1)$ and $\tilde{y}_i\tilde{z}_i=-1$.
Set $s_i=\tilde{y}_i$, which is nonzero.
It follows that
$$
T_{1,i}=S(r_1,s_i),
\qquad
U_{1,i}=N(s_i,\mu_i),
$$
where $S(r_1,s_i)$ and $N(s_i,\mu_i)$ are in
\eqref{local blocks PSN}.

Using the parameters obtained above, set
$$
c=(c_1,\dots,c_{D-1}),
\qquad
s=(s_0,\dots,s_{D-1}),
$$
and
$$
\lambda=(\lambda_1,\dots,\lambda_{D-1}),
\qquad
\mu=(\mu_0,\dots,\mu_{D-1}).
$$
Combining the above comments, we find that the $\widetilde{\mathcal H}$-module $V$ is isomorphic to
$E_D(\rho_-,\rho_+;c,s,\lambda,\mu)$, which is isomorphic to
$E_D(\rho_-,\rho_+;{\bf 1},{\bf 1},\lambda,\mu)$
by Lemma~\ref{even parameter}.

Set ${\bf v}=(\rho_-,\rho_+;\lambda,\mu)$.
For $1\leq i\leq D-1$, since $y_i z_i\neq0$,
\eqref{lambda neq r_0} implies that
$\lambda_i\in\mathbb{C}^{\times}\setminus\{r_0,r_0^{-1}\}$.
Since $V$ is multiplicity-free,
$\lambda_1,\dots,\lambda_{D-1}$ are mutually distinct, and
$\mu_0,\dots,\mu_{D-1}$ are mutually distinct.
Together with
$\rho_-,\rho_+\in\{r_0,r_0^{-1}\}$, this shows that
${\bf v}\in\mathcal{E}_D$.
Since $V\in\mathcal{V}$ and $V$ is isomorphic to
$E_D(\rho_-,\rho_+;{\bf 1},{\bf 1},\lambda,\mu)$, we have
$$
\varphi({\bf v})
=
[(\rho_-,\rho_+;\lambda,\mu)]
=
\mathcal{V}.
$$
Thus $\varphi$ is surjective.

We next show that $\varphi$ is injective.
Let
${\bf v}=(\rho_-,\rho_+;\lambda,\mu)$ and ${\bf v}'=(\rho_-',\rho_+';\lambda',\mu')$
be elements of $\mathcal{E}_D$ such that
$\varphi({\bf v})=\varphi({\bf v}')$.
Abbreviate
$$
W=E_D(\rho_-,\rho_+;{\bf 1},{\bf 1},\lambda,\mu),
\qquad
W'=E_D(\rho_-',\rho_+';{\bf 1},{\bf 1},\lambda',\mu'),
$$
and write
$$
\{b_0,a_1,b_1,\dots,a_{D-1},b_{D-1},a_D\},
\qquad
\{b_0',a_1',b_1',\dots,a_{D-1}',b_{D-1}',a_D'\}
$$
for the adapted block bases of $W$ and $W'$, respectively.
Since $\varphi({\bf v})=\varphi({\bf v}')$, there exists an
$\widetilde{\mathcal H}$-module isomorphism
$\psi: W \to W'$.
Since $\mathbb{C}b_0$ is a $B$-eigenspace with nonzero eigenvalue that is invariant under $t_0$, so is $\mathbb{C}\psi(b_0)$.
Since $\mathbb{C}b_0'$ is the only such $B$-eigenspace of $W'$, we have
\begin{equation}\label{eq:psi.b0}
\psi(b_0)=\gamma b_0'
\end{equation}
for some $\gamma\in\mathbb{C}^{\times}$.
Applying $B$ and $t_0$ to both sides of \eqref{eq:psi.b0}, respectively, we obtain
$$
\mu_0=\mu_0',
\qquad
\rho_-=\rho_-'.
$$
We now proceed inductively.
Suppose that
\begin{equation}\label{eq:psi.bi}
\psi(b_i)=\gamma b_i'
\end{equation}
for some $0\leq i\leq D-1$.
Applying $B$ to \eqref{eq:psi.bi} gives
$\mu_i=\mu_i'$.
Since the two-dimensional $t_1$-blocks with respect to
$\{b_i,a_{i+1}\}$ and $\{b_i',a_{i+1}'\}$ are both $S(r_1,1)$, we have
\begin{equation}\label{eq:t1.bi}
t_1b_i=k(r_1)b_i-a_{i+1},
\qquad
t_1b_i'=k(r_1)b_i'-a_{i+1}'.
\end{equation}
Using \eqref{eq:psi.bi}, \eqref{eq:t1.bi}, and the fact that $\psi$ commutes with the action of $t_1$, we obtain
\begin{equation}\label{eq:psi.ai+1}
\psi(a_{i+1})=\gamma a_{i+1}'.
\end{equation}
Suppose now that $i\leq D-2$.
Using \eqref{eq:psi.ai+1} and the fact that $\psi$ commutes with the action of $A$, we obtain
$\lambda_{i+1}=\lambda_{i+1}'$.
Thus the two-dimensional $t_0$-blocks with respect to
$\{a_{i+1},b_{i+1}\}$ and
$\{a_{i+1}',b_{i+1}'\}$ are both
$R(r_0,\lambda_{i+1},1)$.
Since the lower-left entry of this matrix is nonzero, \eqref{eq:psi.ai+1} and the fact that $\psi$ commutes with the action of $t_0$ give
$\psi(b_{i+1})=\gamma b_{i+1}'$.
Therefore by induction,
$$
\lambda_i=\lambda_i'
\quad (1\leq i\leq D-1),
\qquad 
\mu_i=\mu_i'
\quad (0\leq i\leq D-1),
$$
and
$\psi(a_D)=\gamma a_D'$.
Applying $t_0$ to this equation gives
$\rho_+=\rho_+'$.
Combining these observations, we have ${\bf v}={\bf v}'$.
Thus $\varphi$ is injective.

Consequently, $\varphi$ is bijective.

(ii) For $D=0$, the assertion is immediate.
For $D\geq 1$, the proof is similar to (i).
\end{proof}

\section{Examples}\label{Sec:ex}

In this section,
we investigate some examples of finite-dimensional irreducible $\widetilde{\mathcal{H}}$-modules.
The first example concerns an even-dimensional multiplicity-free $\widetilde{\mathcal H}$-module arising in a combinatorial setting.

\begin{example}\label{ex:dual-polar}
{\rm
A subfamily of the even-dimensional modules $E_D$
constructed in Subsection~\ref{even subsec} arises naturally from dual polar graphs.
We show that the $\widetilde{\mathcal{H}}$-module associated with a dual polar graph in
\cite{LeeTanaka} belongs to this subfamily.

Let $\Gamma$ be a dual polar graph with vertex set $X$ and diameter
$D\geq 3$ over a finite field of order $q$, and let $e$ denote the
parameter associated with $\Gamma$; see \cite[Section~9.4]{BCN}.
Fix a vertex $x\in X$ and a maximal clique $C$ containing $x$.
For $0\leq i\leq D$, let $\Gamma_i(x)$ denote the set of vertices at distance $i$ from $x$.
For $0\leq i\leq D-1$, let $C_i$ denote the set of vertices at distance
$i$ from $C$, and define
$$
C_i^-=\Gamma_i(x)\cap C_i,
\qquad
C_i^+=\Gamma_{i+1}(x)\cap C_i.
$$
Let $\widehat{C}_i^-$ and $\widehat{C}_i^+$ denote the characteristic
vectors of $C_i^-$ and $C_i^+$, respectively.
Let $W$ be the subspace spanned by these vectors.
By \cite{LeeTanaka}, $W$ is $2D$-dimensional with the ordered basis
$$
\mathcal{C}
=
\{
\widehat{C}_0^-,\widehat{C}_0^+,
\widehat{C}_1^-,\widehat{C}_1^+,
\dots,
\widehat{C}_{D-1}^-,\widehat{C}_{D-1}^+
\}.
$$
Set
$$
b_i=\widehat{C}_i^-,
\qquad
a_{i+1}=\widehat{C}_i^+
\qquad (0\leq i\leq D-1).
$$
Then
$$
\mathcal{C} = \{b_0,a_1,b_1,a_2,\dots,b_{D-1},a_D\}.
$$

Now, fix $\epsilon\in\mathbb{C}^{\times}$ such that
$\epsilon^2=-q^{-D-e}$, and set
$$
r_0=q^{-\frac{D}{2}},
\qquad
r_1=\epsilon q^{\frac{D}{2}},
\qquad
\rho_-=\rho_+=q^{-\frac{D}{2}}.
$$
For $1\leq i\leq D-1$, define
$$
\lambda_i=q^{\frac{D}{2}-i},
\qquad
c_i=q^{D-i}-1,
$$
and, for $0\leq i\leq D-1$, define
$$
\mu_i=q^{-i},
\qquad
s_i=-\epsilon^{-1}q^{-\frac{D}{2}}.
$$
Substituting these parameters into the block matrices
$T_0,U_0,T_1,U_1$ defined in \eqref{e.mat:T0}--\eqref{e.mat:U1} 
yields precisely the matrices used in \cite{LeeTanaka} 
to represent $t_0,u_0,t_1,u_1$, respectively, on the primary 
module $W$ with respect to $\mathcal{C}$.

Recall that $A=u_0t_0$ and $B=t_1u_1$.
With respect to $\mathcal{C}$, the matrices representing $A$ and $B$ are
$$
{\rm diag}
\bigl(
	0,q^{\frac{D}{2}-1},
	0,q^{\frac{D}{2}-2},
	\dots,
	0,q^{-\frac{D}{2}+1},
	0,q^{-\frac{D}{2}}
\bigr)
$$
and
$$
{\rm diag}
\bigl(
	1,0,
	q^{-1},0,
	q^{-2},0,
	\dots,
	q^{-D+1},0
\bigr),
$$
respectively.
Therefore
\begin{equation}\label{jointAB-containments}
a_i\in W_{q^{\frac{D}{2}-i},0}
\qquad (1\leq i\leq D),
\qquad
b_i\in W_{0,q^{-i}}
\qquad (0\leq i\leq D-1).
\end{equation}
Since $q>1$, the pairs
$(q^{\frac{D}{2}-i},0)$ for $1\leq i\leq D$ and
$(0,q^{-i})$ for $0\leq i\leq D-1$ are mutually distinct.
Since both $A$ and $B$ are diagonal with respect to $\mathcal{C}$,
it follows that every nonzero joint eigenspace of $(A,B)$ is
one-dimensional.
Therefore $W$ is $(A,B)$-multiplicity-free.
Moreover, the containments in \eqref{jointAB-containments}, together with the block forms
of $T_0$ and $T_1$, show that $\mathcal{C}$ is an even adapted block basis.

The scalars $\lambda_1,\dots,\lambda_{D-1}$ are mutually distinct and satisfy
$\lambda_i\notin \{r_0,r_0^{-1}\}$ for $1\leq i\leq D-1$.
The scalars $\mu_0,\dots,\mu_{D-1}$ are mutually distinct, and none of
$c_1,\dots,c_{D-1},s_0,\dots,s_{D-1}$ is zero.
Hence Theorem~\ref{even irreducible} implies that $W$ is an irreducible $\widetilde{\mathcal{H}}$-module.

Consequently, $W$ is isomorphic to the module $E_D$ corresponding to
the parameters specified above.
Thus the construction in \cite{LeeTanaka} provides a combinatorial
realization of a member of the even-dimensional family classified in
Theorem~\ref{thm:main}.
}
\end{example}

The next example shows that there is a
finite-dimensional irreducible $\widetilde{\mathcal{H}}$-module
which is not multiplicity-free.

\begin{example}
{\rm
Let $V = F \oplus E$,
where $F\cong \mathbb{C}^2 \cong E$.
Let $\{a_1,a_2\}$ and $\{b_1,b_2\}$ be ordered bases of $F$ and $E$, respectively.
Then we have an ordered basis $\mathfrak{B}:=\{a_1,a_2,b_1,b_2\}$ of $V$.
Choose $r_0,r_1, \mu \in \mathbb{C}^{\times}$ with $r_0^2 \neq 1$.
Set 
\begin{center}
$P=\begin{pmatrix}r_0&0\\0&r_0^{-1}\end{pmatrix}$, \quad
$G=\begin{pmatrix}1&1\\1&2\end{pmatrix}$
\end{center}
and let $Q=GPG^{-1}$.
Now, define the block matrices
$T_0,U_0,T_1,U_1$ as follows:
\begin{center}
$T_0 = \begin{pmatrix}P&0\\0&Q\end{pmatrix}$, \quad
$U_0 = \begin{pmatrix}0&0\\0&I_2\end{pmatrix}$,\\[5pt]
$T_1 = \begin{pmatrix}(r_1+r_1^{-1})I_2&-r_1^{-1}I_2\\ r_1I_2&0\end{pmatrix}$, \quad
$U_1 = \begin{pmatrix}0&0\\-r_1\mu I_2&0\end{pmatrix}$.
\end{center}
A direct computation shows that
$T_0,U_0,T_1,U_1$ satisfy the defining relations in
Definition~\ref{nil.H}.
Hence there exists an $\widetilde{\mathcal{H}}$-module structure on $V$ such that
the matrices representing $t_0,u_0,t_1,u_1$ with respect to $\mathfrak{B}$
are $T_0,U_0,T_1,U_1$, respectively.

Now, we claim that $V$ is irreducible.
Let $W$ be an $\widetilde{\mathcal{H}}$-submodule of $V$.
Let $w \in W$.
Then there exist $x \in F$ and $y \in E$ such that
$w = x+y$.
Hence we obtain
\begin{center}
$y = u_0x + u_0 y = u_0 (x+y) = u_0w \in W$,
\end{center}
so $x \in W$.
This implies that
\begin{center}
$W= W_F \oplus W_E$,
\end{center}
where $W_F := W \cap F$ and $W_E := W \cap E$.
Define $\varphi: W_F \to W_E$ by $a \mapsto u_1a$.
Then it is routine to check that $\varphi$ is a $\mathbb{C}$-vector space isomorphism.
Hence $\dim(W_F) = \dim(W_E)$.

Suppose to the contrary that $\dim(W_F) = 1$.
Then there exist $\alpha, \beta \in \mathbb{C}$ such that
$\{w_F:= \alpha a_1 + \beta a_2\}$ is a basis of $W_F$.
Consider the vector $(\alpha,\beta)^T$
as the coordinate vector of $w_F$ with respect to $\{a_1,a_2\}$.
Since $t_0$ preserves $W_F$, there exists $k \in \mathbb{C}$ such that
$P (\alpha,\beta)^T = k (\alpha,\beta)^T$.
Hence we obtain
\begin{center}
$r_0\alpha a_1 + r_0^{-1} \beta a_2 = k \alpha a_1 + k\beta a_2$.
\end{center}
If both $\alpha$ and $\beta$ are nonzero,
then $r_0 = k = r_0^{-1}$,
which means that $r_0^2 = 1$, a contradiction.
Hence either $\alpha \neq 0$ and $\beta = 0$ or
$\alpha = 0$ and $\beta \neq 0$.
Thus $W_F = \mathbb{C}a_j$ for some $j \in \{1,2\}$.
Since $u_1a_j = -r_1\mu b_j$ and $\varphi$ is an isomorphism,
we have $W_E = \mathbb{C}b_j$.
Let $e_1 = (1,0)^T$ and $e_2 = (0,1)^T$.
As $t_0$ preserves $W_E$,
the vector $e_j$ is an eigenvector of $Q$.
However,
\begin{center}
$Q = \begin{pmatrix}2r_0-r_0^{-1}&-r_0+r_0^{-1}\\2r_0-2r_0^{-1}&-r_0+2r_0^{-1}\end{pmatrix}$,
\end{center}
and $r_0^2 \neq 1$,
so neither $e_1$ nor $e_2$ is an eigenvector of $Q$, a contradiction.
Hence $\dim(W_F) \neq 1$.
Therefore either $W_F=0$ and $W_E=0$, or $W_F=F$ and $W_E=E$.
Thus either $W = 0$ or $W = V$.
Consequently, $V$ is irreducible.

On the other hand,
note that on the $\widetilde{\mathcal{H}}$-module $V$,
the matrices representing $A$ and $B$ with respect to $\mathfrak{B}$
are given as follows:
\begin{center}
$A : \quad {\rm blockdiag}(0,Q)$, \quad $B: \quad {\rm blockdiag}(\mu I_2,0)$,
\end{center}
which means that $V_{0,\mu} = F$,
and hence $\dim(V_{0,\mu}) = 2$.
Thus $V$ is not $(A,B)$-multiplicity-free.
}
\end{example}

The following example shows that there is a multiplicity-free finite-dimensional irreducible
$\widetilde{\mathcal{H}}$-module which does not have an adapted block basis.

\begin{example}
{\rm
Let $V$ be a two-dimensional vector space with an ordered basis $\{b_0,a_1\}$.
Choose $c, r_0,r_1, \mu \in \mathbb{C}^{\times}$.
Define the matrices $T_0,U_0,T_1,U_1$ as follows:
\begin{center}
$T_0 = \begin{pmatrix}r_0&c\\0&r_0^{-1}\end{pmatrix}$,\quad
$U_0 = \begin{pmatrix}1&-cr_0\\0&0\end{pmatrix}$,\\[5pt]
$T_1 = \begin{pmatrix}0&-r_1\\ r_1^{-1}&r_1+r_1^{-1}\end{pmatrix}$,\quad
$U_1 = \begin{pmatrix}0&r_1\mu\\0&0\end{pmatrix}$.
\end{center}
A direct computation shows that
$T_0,U_0,T_1,U_1$ satisfy the defining relations in
Definition~\ref{nil.H}.
Hence there exists an $\widetilde{\mathcal{H}}$-module structure on $V$ such that
the matrices representing $t_0,u_0,t_1,u_1$ with respect to
$\{b_0,a_1\}$ are $T_0,U_0,T_1,U_1$, respectively.
Also, note that on the $\widetilde{\mathcal{H}}$-module $V$,
the matrices representing $A$ and $B$ with respect to $\{b_0,a_1\}$ are given as follows:
\begin{center}
$A : \quad {\rm diag}(r_0,0)$, \quad $B: \quad {\rm diag}(0,\mu)$.
\end{center}
Hence $V$ is $(A,B)$-multiplicity-free.
On the other hand,
suppose to the contrary that $V$ has an adapted block basis, say $\{b_0',a_1'\}$.
By the definition of an even adapted block basis,
$b_0' \in V_{0,\mu_0}$ for some $\mu_0 \in \mathbb{C}^{\times}$.
Since the only nonzero joint eigenspace of the form $V_{0,\beta}$ is $V_{0,\mu}= \mathbb{C}a_1$,
we have $\mu_0 = \mu$ and $b_0' \in \mathbb{C}a_1$.
However,
\begin{center}
$t_0a_1 = cb_0 + r_0^{-1} a_1 \notin \mathbb{C}a_1$
\end{center}
since $c \neq 0$.
Therefore $t_0$ does not preserve $\mathbb{C}b_0'$, a contradiction.
Hence $V$ does not have an adapted block basis.

Now, let $W$ be an $\widetilde{\mathcal{H}}$-submodule of $V$.
Suppose that $\dim(W) = 1$.
Then there exist $\alpha, \beta \in \mathbb{C}$ such that
$\{w := \alpha b_0 + \beta a_1\}$ is a basis of $W$.
Since $u_0$ preserves $\mathbb{C}w$,
there exists $k \in \mathbb{C}$ such that $u_0w = kw$.
Hence we obtain
\begin{center}
$k \alpha b_0 + k \beta a_1 =  kw = u_0w = (\alpha-cr_0 \beta) b_0$.
\end{center}
This implies that $k \beta = 0$.
If $\beta = 0$,
then $w = \alpha b_0$,
so we have
\begin{center}
$t_1w = \alpha t_1 b_0 = \alpha r_1^{-1} a_1 \notin W$,
\end{center}
a contradiction.
Hence $k = 0$.
Therefore we obtain $\alpha = c r_0 \beta$,
so $w = \beta(cr_0b_0 + a_1)$.
This implies that $\{w':= cr_0b_0 + a_1\}$ is a basis of $W$.
Since $u_1 w' \in W$,
there exists $k' \in \mathbb{C}$ such that $u_1w' = k'w'$.
Observe that
\begin{center}
$k'cr_0b_0 + k' a_1 = k'w' = u_1w' = r_1\mu b_0$,
\end{center}
so $k' = 0$,
which means that $r_1 \mu= 0$, a contradiction.
Thus there is no one-dimensional $\widetilde{\mathcal{H}}$-submodule of $V$.
Consequently, $V$ is irreducible.
}
\end{example}

\acknowledgments

J. Park is supported by the National Research Foundation of Korea (NRF) grant funded by the Korea government (MSIT) (RS-2024-00356153).

\end{document}